\documentclass[12pt]{amsart}

\usepackage{amsmath,amssymb,enumerate,mathtools}
\newtheorem{theorem}{Theorem}[section]
\newtheorem{claim}{Claim}[theorem]
\newtheorem{lemma}[theorem]{Lemma}

\theoremstyle{definition}

\newcommand{\bR}{\mathbb R}
\newcommand{\bZ}{\mathbb Z}

\newcommand{\cC}{\mathcal{C}}

\newcommand{\cZ}{\mathcal{Z}}
\newcommand{\cW}{\mathcal{W}}

\newcommand{\eps}{\varepsilon}

\newcommand{\prob}{\mathbf{P}}
\newcommand{\ev}{\mathbf{E}}

\newcommand{\norm}[1]{\left\lVert #1 \right\rVert}

\begin{document}
\sloppy
\allowdisplaybreaks
\title[Cyclicity of a random subgraph]
{On the probability that a random subgraph contains a circuit}\begin{abstract}
	Let $\mu > 2$ and $\eps > 0$. We show that, if $G$ is a sufficiently large simple graph of average degree at least $\mu$, and $H$ is a random spanning subgraph of $G$ formed by including each edge independently with probability $p \ge \tfrac{1}{\mu-1} + \eps$, then $H$ contains a cycle with probability at least $1 - \eps$. 
	\end{abstract}
\author[Nelson]{Peter Nelson}
\date{\today}

\maketitle

\section{Introduction}

We prove the following theorem:

\begin{theorem}\label{main}
	For all $\eps > 0$ and $\mu > 2$, there exists $N \in \bZ$ so that, if $G$ is a graph on at least $N$ vertices with average degree at least $\mu$, and $H$ is a random spanning subgraph of $G$ formed by including each edge of $G$ independently with probability $p \ge \tfrac{1}{\mu-1} + \eps$, then $H$ contains a circuit with probability at least $1-\eps$. 
\end{theorem}

The same result was shown by Alon and Bachmat [\ref{ab06}] in the special case where $G$ is regular. They also show that, for any $\eps > 0$ and integer $d \ge 3$, there is a $d$-regular graph $G$ such that a random spanning subgraph of $G$ in which edges are chosen with probability $\tfrac{1}{d-1} - \eps$ contains a circuit with probability at most $\eps$; in other words, $\tfrac{1}{\mu-1}$ cannot be replaced by any smaller value in the above theorem if $\mu \in \bZ$. On the other hand, it can be shown that if $\mu \notin \bZ$, then the value $\tfrac{1}{\mu-1}$ \emph{can} be lowered; we discuss this in the next section. 

Our proof is based on that of the main theorem of [\ref{nvz15}], which concerns the related question of the probability that $H$ contains at least half of the edges of some circuit of $G$. The aim of this paper is to present the relatively simple proof of Theorem~\ref{main} without the inherent and numerous technicalities of [\ref{nvz15}]. 

Similar questions concerning cycles in random subgraphs of graphs with a given \emph{minimum} degree were considered in [\ref{fk13}] and [\ref{ks14}]. 

\section{Preliminaries}
	All graphs are simple and finite unless otherwise stated. We use some standard graph theory terminology such as \emph{path}, \emph{walk}, \emph{girth} and \emph{adjacency matrix} (for a directed graph); see [\ref{diestel}] for a reference. For $p \in [0,1]$, a \emph{$p$-random} subset of a set $E$ refers to a set $X \subseteq E$ obtained by including each element of $E$ independently at random with probability $p$. We also use some basic probability theory; for a reference, see [\ref{durrett}].  
	
	\subsection*{Zero-one Laws}
Let $E$ be a finite set. For each $\cW \subseteq 2^E$, let $f_{\cW}\colon [0,1] \to [0,1]$ be defined by $f_{\cW}(p) = \sum_{X \in \cW}p^{|X|}(1-p)^{|E|-|X|}$; i.e. $f_{\cW}(p)$ is the probability that a $p$-random subset of $E$ is in $\cW$. An important tool in our proof is a  theorem of Margulis that gives a very general sufficient condition for $f_{\cW}(p)$ to display a zero-one-law type behaviour. In [\ref{margulis74}], Margulis states his theorem in very high generality (and in Russian). Our statement follows a convenient `discrete' formulation found in ([\ref{zemor}], Section 2). 

	We say that $\cW \subseteq 2^{E}$ is \emph{increasing} if, whenever $X \in \cW$ and $X \subseteq X'$, we have $X' \in \cW$. Given $X \in 2^{E}$, let $s(X)$ denote the collection of subsets of $E$ that differ from $X$ by a single addition or removal (i.e. the Hamming sphere of radius $1$ around $X$ in $2^{E}$). Let $\Delta (\cW)$ denote the minimum nonzero value of $|s(X) \setminus \cW|$ over all $W \in \cW$. 
	
	Note that, if $\varnothing \ne \cW \ne 2^E$ and $\cW$ is increasing, the function $f_{\cW}(p)$ is monotonely increasing with $f_{\cW}(0) = 0$ and $f_{\cW}(1) = 1$. Margulis' theorem states that if $\Delta(\cW)$ is sufficiently large, then the value of $f_{\cW}(p)$ is nearly always close to zero or one.
	
	\begin{theorem}\label{margulis}
		For all $\eps > 0$ there exists $s \in \bZ$ so that, if $E$ is a finite set, and $\cW \subseteq 2^E$ is increasing and satisfies $\Delta(W) \ge s$, then the interval $\{p \in [0,1]\colon \eps \le f_{\cW}(p) \le 1-\eps\}$ has length less than $\eps$. 
	\end{theorem}
	
	We will apply this result in the very special case where $E$ is the edge set of a graph $G$, and $\cW$ is the collection of edge-sets of subgraphs of $G$ that contain a circuit. In this setting, it is easy to see that the parameter $\Delta(\cW)$ is exactly the girth of $G$.

\subsection*{Non-backtracking walks}\label{walksection}

A \emph{non-backtracking walk} of length $\ell$ in a graph $G$ is a walk $(v_0,v_1,\dotsc,v_\ell)$ of $G$ so that $v_{i+1} \ne v_{i-1}$ for all $i \in \{1,\dotsc, \ell-1\}$. In all nontrivial cases, the number of such walks grows roughly exponentially in $\ell$; in this section we state a result of Alon et al. that estimates the base of this exponent. 
Let $G = (V,E)$ be a connected graph of minimum degree at least $2$. 
Let $\overline{E} = \{(u,v) \in V^2: u \sim_G v\}$ be the $2|E|$-element set of arcs of $G$. Let $B = B(G) \in \{0,1\}^{\overline{E} \times \overline{E}}$ be the matrix so that $B_{(u,v),(u',v')} = 1$ if and only if $u' = v$ and $u \ne v'$. It is easy to see that 
\begin{enumerate}
	\item\label{scd} $B$ is the adjacency matrix of a strongly connected digraph (essentially the `line digraph' of $G$), and 
	\item\label{nbw} For each integer $\ell \ge 2$, the entry $(B^{\ell-1})_{e,f}$ is  the number of non-backtracking walks of length $\ell$ in $G$ with first arc $(v_0,v_1) = e$ and last arc $(v_\ell,v_{\ell+1})= f$.
\end{enumerate}

 By (\ref{scd}) and the Perron-Fr\"obenius theorem (see [\ref{gr}], section 8.8), there is a positive real eigenvalue $\lambda_*$ of $B$ and an associated positive real eigenvector $w_*$, so that $|\lambda_*| \ge |\lambda|$ for every eigenvalue $\lambda$ of $B$. Furthermore, by Gelfand's formula [\ref{gelfand}] we have $\lambda_* = \lim_{n \to \infty} \norm{B^n}^{1/n}$, where $\norm{B^n}$ denotes the sum of the absolute values of the entries of $B^n$. By (\ref{nbw}), the parameter $\lambda_* = \lambda_*(B(G))$ thus governs the growth of non-backtracking walks in $G$. It is clear that if $G$ is $d$-regular we have $\lambda_*(B(G)) = d-1$; the following result of Alon et al. [\ref{alon}] shows that the average degree gives a similar lower bound for general graphs:
 \begin{lemma}\label{boundLambda}
 	Let $\mu \ge 2$. If $G$ is a connected graph of average degree at least $\mu$ and minimum degree at least $2$, then $\lambda_*(B(G)) \ge \mu - 1$. 
 \end{lemma}
 In fact, the argument in [\ref{alon}] shows that $\lambda_*(B(G)) \ge \Lambda(G)$, where $\Lambda(G)$ is a certain symmetric function in the degree sequence of $G$ that is bounded below by $\mu-1$. When $\mu \notin \bZ$, the inequality $\Lambda(G) \ge \mu - 1$ cannot hold with equality; in fact one can show (see [\ref{nvz15}], Lemma 3.2) that $\lambda_*(B(G)) \ge \mu - 1 + \tfrac{\eta(\mu)^3}{8\mu^3}$, where $\eta(\mu)$ denotes the distance from $\mu$ to the nearest integer. This can easily be shown to lead to an improved version of Theorem~\ref{main} where $\frac{1}{\mu-1}$ is replaced by a strictly smaller value for nonintegral $\mu$. 
 
\section{Covering Trees}

Given a connected graph $G = (V,E)$ of minimum degree at least $2$, we denote the set of arcs of $G$, as before, by $\bar{E}$. Given an arc $e_0  = (u_0,u_1) \in \bar{E}$, the \emph{covering tree of $G$ at $e_0$}, for which we write $\Gamma_{e_0}(G)$, is the infinite rooted tree $\Gamma$ whose root is the length-zero walk $(u_0)$, the other vertices are the non-backtracking walks of $G$ whose first arc is $e_0$, and the children of each walk $(u_0,u_1, \dotsc,u_k) \in V(\Gamma)$ are exactly its extensions by a single arc: that is, the nonbacktracking walks of the form $(u_0,u_1, \dotsc, u_k,u_{k+1})$. Note that the unique child of the root is the walk $(e_0) = (u_0,u_1)$.

Given such a $\Gamma$ with root $r$, we borrow some terminology from [\ref{lyons}]. For $x \in V(\Gamma)$, we write $|x|$ for the distance from $x$ to $r$ in $\Gamma$. For $x,y \in V(\Gamma)$, we write $x \preceq y$ if $x$ is on the path from $r$ to $y$, and $x \wedge y$ for the \emph{join} of $x$ and $y$ in $\Gamma$: that is, the unique vertex $z$ on both the path from $r$ to $x$ and the path from $r$ to $y$ for which $|z|$ is maximized. 

The map $\theta\colon V(\Gamma) \to V(G)$ that assigns each walk to its final vertex is a graph homomorphism that is injective when restricted to the neighbourhood of any vertex. To analyse the probability that a $p$-random subset of $E(G)$ contains a circuit, we consider the probability that a $p$-random subset of $E(\Gamma)$ contains a long path containing the root. The following lemma is a stronger `constructive' version of Theorem 6.2 of [\ref{lyons}] (which applies to general infinite rooted trees) in the case where $\Gamma$ is the covering tree of a finite graph. 

\begin{lemma}\label{treesubset}
	Let $G$ be a connected graph of minimum degree at least $2$ and let $\lambda = \lambda_*(B(G))$. There is an arc $e_0$ of $G$ so that, if $p \in [0,1]$ and $X$ is a $p$-random subset of $E(\Gamma_{e_0}(G))$, then, for each integer $n \ge 1$, the probability that $X$ contains a $n$-edge path of $\Gamma_{e_0}(G)$ containing the root is at least $p - \tfrac{1}{\lambda}$.
	\end{lemma}
\begin{proof}

	Let $B = B(G)$, and let $w \in \bR^{\overline{E}(G)}$ be the (positive, real) eigenvector of $B(G)$ corresponding to $\lambda$ whose largest entry is $1$. Let $e_0 \in \overline{E}(G)$ be such that $w(e_0) = 1$; we show that $e_0$ satisfies the lemma. We may assume that $p > \frac{1}{\lambda}$. Let $\Gamma = \Gamma_{e_0}(G)$, let $r$ be the root of $G$, and let $\pi\colon V(\Gamma) \setminus \{r\} \to \overline{E}(G)$ be the map associating each walk with its last arc. 
	
	Let $\phi\colon V(\Gamma) \to \bR_{> 0}$ be defined by $\phi(r) = 1$ and $\phi(v) = \lambda^{1-|v|}w(\pi(v))$ for all $v \ne r$. Note that $\phi((e_0)) = w(e_0) = 1$ and that, since $\lambda$ is an eigenvalue of $B$, the sum of $\phi(x)$ over the children $x$ of a vertex $v$ (or over the descendants $x$ of $v$ at any fixed level) is equal to $\phi(v)$. In other words, $\phi$ is a \emph{unit flow} of $\Gamma$.
	
	For $X \subseteq E(\Gamma)$, let $R_X$ denote the vertex set of the component of $\Gamma[X]$ containing $r$. Define a random variable $Q = Q(X)$ by	
	\[Q = p^{-n}\sum_{|x| = n}\phi(x)1_{R_X}(x).\]
	Since $\sum_{|x| = n}\phi(x) = 1$ and each vertex at distance $n$ from the root is in $R_{X}$ with probability exactly $p^n$, we have $\ev(Q) = 1$. We now bound the variance of $Q$.
	\begin{claim}
		$\ev(Q^2) \le (p-\tfrac{1}{\lambda})^{-1}$.
	\end{claim}
	\begin{proof}
		Note that $|x \wedge y| \ge 1$ whenever $|x|,|y| \ne 0$. We have 
		\begin{align*}\ev(Q^2) &= p^{-2n}\sum_{|x| = |y| = n}\phi(x)\phi(y)\prob(x,y \in R_X)\\
		&= p^{-2n}\sum_{|x| = |y| = n}\phi(x)\phi(y) p^{2n-|x \wedge y|}\\
		&= \sum_{|x| = |y| = n}\phi(x)\phi(y)p^{-|x \wedge y|}\\
		&= \sum_{1 \le |z| \le n} p^{-|z|}\sum_{\substack{|x|=|y|= n\\ x \wedge y = z}}\phi(x)\phi(y)\\
		&\le \sum_{1 \le |z| \le n}p^{-|z|}\left(\sum_{\substack{|x| = n \\ x \succ z}}\phi(x)\right)^2\\
		&= \sum_{1 \le |z| \le n}p^{-|z|}\phi(z)^2\\
		&= \sum_{i=1}^n p^{-i} \sum_{|z| = i}\phi(z)^2
		\end{align*}
	Now by definition of $\phi$ and the fact that $w(e) \le 1$ for all $e$ we have
	\[\phi(z)^2 = \lambda^{2-2|z|}w(\pi(z))^2 \le \lambda^{2-2|z|}w(\pi(z)).\]
	For each $e \in \overline{E}$, let $b_e \in \bR^{\overline{E}}$ be the corresponding standard basis vector. For each $1 \le i \le n$, the number of $z \in V(\Gamma)$ with $|z| = i$ and $\pi(z) = e$ is exactly $b_{e_0}^TB^{i-1} b_e$, so it follows from $Bw = \lambda w$ and $b_{e_0}^Tw = 1$ that 
	\[\sum_{|z| = i}\phi(z)^2 \le \lambda ^{2-2i}b_{e_0}^TB^{i-1} \sum_{e \in \overline{E}} b_e w(e) = \lambda^{2-2i}b_{e_0}^TB^{i-1}w = \lambda^{1-i}.\]
	Therefore $\ev(Q^2) \le \sum_{i = 1}^n p^{-i}\lambda^{1-i} < (p-\tfrac{1}{\lambda})^{-1}$, since $p \lambda > 1$.
	\end{proof}
	The Cauchy-Schwartz inequality now gives
	\[1 = \ev(Q)^2 = \ev(Q \cdot 1_{Q > 0})^2 \le \ev(Q^2)\prob(Q > 0) \le (p - \tfrac{1}{\lambda})^{-1}\prob(Q > 0),\]
	so $\prob(Q > 0) \ge p - \tfrac{1}{\lambda}$. If $Q > 0$ then $X$ contains an $n$-edge path of $\Gamma$ containing the root; the lemma follows. 
	\end{proof}

	\section{Circuits}
	
	In this section, we prove our main theorem. First we show that, if $G$ is `non-degenerate' and $\lambda = \lambda_*(B(G))$, then a $(\tfrac{1}{\lambda}+\eps)$-random subset of $E(G)$ contains a circuit with non-negligible probability.
	
	\begin{theorem}\label{mainalg}
		Let $G$ be a connected graph with minimum degree at least $2$ and let $\lambda = \lambda_*(B(G))$. Let $\eps > 0$ and $p \in [\tfrac{1}{\lambda} + \eps,1]$. If $X$ is a $p$-random subset of $E(G)$, then $X$ contains a circuit of $G$ with probability at least $\tfrac{\eps^2}{4}$. 
	\end{theorem}
	\begin{proof}
		Let $p_1 = p-\tfrac{\eps}{2} \ge \lambda^{-1} + \tfrac{\eps}{2}$, and let $p_2 \ge \tfrac{\eps}{2}$ be such that $1-p = (1-p_1)(1-p_2)$. Let $X_1$ be a $p_1$-random subset of $E(G)$ and $X_2$ be a $p_2$-random subset of $E(G)$ independent of $X_1$; note that $X_1 \cup X_2$ is identically distributed to a $p$-random subset of $E(G)$. 
				
		Let $e_0 = (u,v)$ be an arc of $G$ given by Lemma~\ref{treesubset} for $p_1$ and $\lambda$ and let $\Gamma = \Gamma_e(G)$. Let $\theta \colon V(\Gamma) \to V(G)$ denote the natural homomorphism from $\Gamma$ to $G$. For each set $U \subseteq V(G)$, let $G(U)$ denote the subgraph of $G$ induced by $U$. 
		
		Given $X_1$, let $R(X_1)$ denote the set of vertices $x$ of $G$ for which there is a non-backtracking walk of $G[X_1]$ from $u$ to $x$ that either contains no arc (that is, $x=u$), or has first arc $(u,v)$. Observe that either $R(X_1) = \{u\}$, or $R(X_1)$ is the vertex set of a connected subgraph of $G$ containing the edge $uv$. Similarly, for a $p_1$-random subset $Y_1$ of $E(\Gamma)$, let $R(Y_1)$ denote the vertex set of the component of $\Gamma[Y_1]$ containing the root. Observe that $G(R(X_1))$ and $G(\theta(R(Y_1)))$ are both connected subgraphs of $G$ containing $u$ (in general these graphs have edges not in $X_1$ or $\theta(Y_1)$). Let $C_G$ denote the event that $G(R(X_1))$ contains a circuit, and $C_{\Gamma}$ the event that $G(\theta(R(Y_1)))$ contains a circuit. 
		
		\begin{claim}
			$\prob(C_G) = \prob(C_{\Gamma})$. 
		\end{claim}
		\begin{proof}
			Let $\cZ'$ denote the collection of subsets of $V(G)$ that induce an acyclic connected subgraph of $G$ containing $u$ and $v$, and let $\cZ = \cZ' \cup \{\{u\}\}$. The event $C_G$ fails to hold exactly when $R(X_1) \in \cZ$, so 
			\[1 - \prob(C_G) = \sum_{Z \in \cZ} \prob(R(X_1) = Z).\]
			Similarly, we have
			\[1 - \prob(C_{\Gamma}) = \sum_{Z \in \cZ} \prob(\theta(R(Y_1)) = Z).\]		
		Clearly $\prob(R(X_1) = \{u\}) = \prob(\theta(R(Y_1)) = \{u\}) = 1-p$. Let $Z \in \cZ'$. By acyclicity of $G(Z)$, there is a unique subtree $\Gamma_Z$ of $\Gamma$ that contains the root of $\Gamma$ and satisfies $\theta(V(\Gamma_Z)) = Z$, and moreover $G(Z)$ and $\Gamma_Z$ are isomorphic finite trees. Now $|E(G(Z))| = |E(\Gamma_Z)|$, and the number of edges of $G$ with exactly one end in $Z \setminus \{u\}$ is equal to the number of edges of $\Gamma$ with exactly one end in $V(\Gamma_Z)$, so \[\prob(R(X_1) = Z) = \prob(R(Y) = V(\Gamma_Z)) = \prob(\theta(R(Y)) = Z).\] The claim now follows from the above two summations.  
		\end{proof}
	If $Y_1$ contains a $|V(G)|$-edge path of $\Gamma$ that contains the root, then this path is mapped by $\theta$ to a non-backtracking walk of $G$ visiting some vertex twice, so $G(\theta(R(Y_1)))$ contains a circuit of $G$. Thus, by the claim above and Lemma~\ref{treesubset}, we have $\prob(C_G) = \prob(C_\Gamma) \ge p_1-\lambda \ge \tfrac{\eps}{2}$. 
	
	Suppose that $C_G$ holds; then $H = G(R(X_1))$ is a graph containing a cycle and having a spanning tree contained in $X_1$. Thus, there is some edge $f$ of $H$ such that $X_1 \cup \{f\}$ contains a circuit of $H$; such an $f$ exists for every $X_1$ satisfying $C_G$. Now $f \in X_2$ with probability $p_2$, so the probability that $X_1 \cup X_2$ contains a circuit of $G$ is at least $p_2 \prob(C_G) \ge \left(\frac{\eps}{2}\right)^2 = \frac{\eps^2}{4}$. This gives the result.
		\end{proof}
		
		We now restate and prove our main theorem. 
	
		\begin{theorem}\label{maintech}
			For all $\eps > 0$ and $\mu > 2$, there exists $N= N(\eps,\mu) \in \bZ$ so that, if $G$ is a graph on at least $N$ vertices with average degree at least $\mu$, and $X$ is a $p$-random subset of $E(G)$ for some $p \in [\tfrac{1}{\mu-1} + \eps,1]$, then $X$ contains a circuit of $G$ with probability at least $1-\eps$. 
		\end{theorem}
		\begin{proof}		
			Let $\eps > 0$, $\mu > 2$, and $p \in [\tfrac{1}{\mu-1}+\eps,1]$. Let $s = s(\tfrac{\eps^2}{16})$ be given by Theorem~\ref{margulis}. Let $p_1 = p-\tfrac{\eps^2}{16}$. Let $t$ be an integer so that $(1-p_1^s)^t \le \eps$. 
 Let $\mu_1 \in (2,\mu)$ be such that $\tfrac{1}{\mu_1-1} + \tfrac{\eps}{2} \le p_1$. Set $N = \left\lceil \tfrac{2st}{\mu_1-\mu}\right\rceil$.		
			
			Let $G$ be a graph on at least $N$ vertices with average degree at least $\mu$. For each $x \in [0,1]$, let $f(x)$ denote the probability that an $x$-random subset of $E(G)$ contains a circuit of $G$. We will show that $f(p) \ge 1- \eps$. 
			
			\begin{claim}
				Either
				\begin{itemize}
				\item $f(p_1) \ge 1-\eps$, or
				\item $G$ has a connected subgraph $H$ with girth at least $s$, minimum degree at least $2$, and average degree at least $\mu_1$. 
				\end{itemize}
			\end{claim}
			\begin{proof}[Proof of claim:]
				Let $\cC$ be a maximal collection of pairwise edge-disjoint circuits of size less than $s$ in $G$. If $|\cC| \ge t$ then $f(p_1) \ge 1-(1-{p_1}^s)^{t} \ge 1 - \eps.$ 
				If $|\cC| < t$, then let $G'$ be obtained by removing the edges of all circuits in $\cC$ from $G$. By the maximality of $\cC$, the graph $G'$ has girth at least $s$. Now by choice of $N \le |V(G)|$, we have 
				\[|E(G')| \ge |E(G)| - st \ge \tfrac{\mu}{2}|V(G)| - \tfrac{\mu-\mu_1}{2}N \ge \tfrac{\mu_1}{2}|V(G)|,\]			
			so $G'$ has average degree at least $\mu_1$. Now, let $G''$ be obtained from $G'$ by deleting degree-$1$ vertices until no more such deletions are possible. It is easy to see (since $\mu_1 \ge 2$) that $G''$ has minimum degree at least $2$ and average degree at least $\mu_1$. Any connected component $H$ of $G''$ with largest-possible average degree will satisfy the claim. 			
			\end{proof}
			Since $f(p) \ge f(p_1)$, we may assume that the above subgraph $H$ exists. Let $\cW$ be the collection of edge sets of subgraphs of $H$ that contain a circuit; let $g(x) = \sum_{X \in \cW}x^{|X|}(1-x)^{|E(H)|-|X|}$. This is the probability that an $x$-random subset of $E(H)$ contains a circuit, so we have $f(x) \ge g(x)$. By Lemma~\ref{boundLambda} we have $\lambda_*(B(H)) \ge \mu_1 - 1$, so $p_1 \ge (\lambda_*(B(H)))^{-1} + \tfrac{\eps}{2}$ and therefore $g(p_1) \ge \tfrac{\eps^2}{16}$ by Theorem~\ref{mainalg}.

			 Now $\Delta(\cW)$ is the girth of $H$, so $\Delta(\cW) \ge s$ and therefore the interval $\{x \in [0,1]: \tfrac{\eps^2}{16} \le g(x) \le 1 - \tfrac{\eps^2}{16}\}$ has length at most $\tfrac{\eps^2}{16}$ by Theorem~\ref{margulis}. Therefore $g(p) = g(p_1 + \tfrac{\eps^2}{16}) \ge 1- \tfrac{\eps^2}{16} \ge 1-\eps$. Since $f(p) \ge g(p)$, the result follows.  
\end{proof}

\section*{References}

\newcounter{refs}

\begin{list}{[\arabic{refs}]}
{\usecounter{refs}\setlength{\leftmargin}{10mm}\setlength{\itemsep}{0mm}}

\item\label{ab06}
N. Alon and E. Bachmat, 
Regular graphs whose subgraphs tend to be acyclic,
Random Struct. Algo. 29 (2006), 324--337.

\item\label{alon}
N. Alon, S. Hoory and N. Linial,
The Moore Bound for Irregular Graphs,
Graph Combinator. 18 (2002), 53--57. 

\item\label{dz}
L. Decreusefond and G. Z\'{e}mor,
On the error-correcting capabilities of cycle codes of graphs,
Combin. Probab. Comput. 6 (1997), 27--38. 

\item\label{diestel}
R. Diestel, 
Graph Theory, 
Springer, 2000. 

\item\label{durrett}
R. Durrett, 
Probability: Theory and Examples (4th edition), 
Cambridge University Press, 2010. 

\item\label{fk13}
A. Frieze and M. Krivelevich, 
On the non-planarity of a random subgraph,
Combin. Probab. Comput. 22 (2013), 722--732.

\item\label{gr}
C. Godsil and G. Royle, 
Algebraic Graph Theory,
Springer, 2001. 

\item\label{gelfand}
I. Gelfand,
Normierte ringe,
Rech. Math. [Mat. Sbornik] N.S., 9 (1941), 3--24

\item\label{ks14}
M. Krivelevich and W. Samotij, 
Long paths and cycles in random subgraphs of H-free graphs,
Electron. J. Combin. 21 (2014), P1.30. 

\item\label{lyons}
R. Lyons, 
Random walks and percolation on trees,
Ann. Probab. 18 (1990), 931--958.

\item\label{margulis74}
G. Margulis, 
Probabilistic characterisations of graphs with large connectivity, 
Problemy Peredachi Informatsii 10 (1974), 101--108. 

\item\label{nvz15}
P. Nelson and Stefan H.M. van Zwam, 
On the maximum-likelihood decoding threshold for graphic codes, 
In preparation.

\item\label{zemor}
G. Z\'{e}mor, 
Threshold effects in codes, 
In \emph{Algebraic Coding: Lecture Notes in Computer Science 781.} Springer-Verlag, 278--286. 
\end{list}

\end{document}